\documentclass[journal]{IEEEtran}
\usepackage{etex}
\usepackage[vlined,ruled]{algorithm2e}
\usepackage[dvipsnames]{xcolor}
\usepackage{pgfpages}
\usepackage[cmex10]{amsmath}		%Load Fonts
\usepackage{amssymb, mathrsfs}
\usepackage{cite}
\usepackage{url}

\usepackage{dsfont}
\usepackage{siunitx}
\usepackage{verbatim}									%verbatim and comment environemtns		
\usepackage{mathtools}									%used for xmapsto command
\usepackage{siunitx}									%load SI units
\usepackage{mdframed}

\usepackage{graphicx}
\usepackage{epstopdf}

\usepackage[font=small]{caption}
\usepackage{subcaption}
\usepackage[activate={true,nocompatibility},final,tracking=true,kerning=true,spacing=true,factor=1100,stretch=10,shrink=10]{microtype}

\usepackage{array}
\usepackage{multirow}
\usepackage{enumerate}
\usepackage{booktabs}
\usepackage{arydshln}

\graphicspath{{./fig/}}

\usepackage{soul}

\usepackage{pgfplots}
\pgfplotsset{width=8cm,compat=newest}
\newcommand*{\damping}{0.006}%
\newcommand*{\freq}{25}%
\pgfmathsetmacro{\freqd}{sqrt(1-(\damping)^2)*\freq}%

\pgfplotsset{
    standard/.style={
    axis x line=middle,
    axis y line=middle,
    enlarge x limits=0.15,
	enlarge y limits=0.15,
	every axis plot post/.style={mark options={fill=black}},
	}
}

\pgfplotsset{%
    ,compat=1.12
    ,every axis x label/.style={at={(current axis.right of origin)},anchor=north west}
    ,every axis y label/.style={at={(current axis.above origin)},anchor=north east}
    }

\usepackage{tikz}
\usepackage{tikz-cd}
\usetikzlibrary{arrows}
\usetikzlibrary{decorations.pathmorphing}
\usetikzlibrary{decorations.markings}
\usetikzlibrary{snakes}
\usetikzlibrary{matrix}
\usetikzlibrary{calc}
\usetikzlibrary{patterns}
\usetikzlibrary{positioning}
\usetikzlibrary{shapes}
\usetikzlibrary{shapes.symbols}
\usetikzlibrary{shapes.geometric}
\usetikzlibrary{shadows.blur}
\usetikzlibrary{shapes.arrows}
\usetikzlibrary{shapes.multipart}
\usetikzlibrary{shapes.callouts}
\usetikzlibrary{shapes.misc}

\tikzstyle{every node}=[font=\small]
\tikzstyle{every path}=[line width=0.8pt,line cap=round,line join=round]

\usepackage[american,smartlabels]{circuitikz}
\ctikzset{bipoles/thickness=1}
\ctikzset{bipoles/length=0.8cm}
\ctikzset{bipoles/diode/height=.375}
\ctikzset{bipoles/diode/width=.3}
\ctikzset{tripoles/thyristor/height=.8}
\ctikzset{tripoles/thyristor/width=1}
\ctikzset{bipoles/vsourceam/height/.initial=.7}
\ctikzset{bipoles/vsourceam/width/.initial=.7}

\newcommand{\real}{\mathbb{R}}

\newcommand{\setdef}[2]{\left\lbrace#1 \;\left|\; #2\right.\right\rbrace}

\DeclareMathOperator*{\minimize}{minimize} 									% minimize
\newcommand{\vect}[1]{\mathbbold{#1}}
\newcommand{\vones}[1][]{\vect{1}_{#1}}
\newcommand{\vzeros}[1][]{\vect{0}_{#1}}
\DeclareSymbolFont{bbold}{U}{bbold}{m}{n}
\DeclareSymbolFontAlphabet{\mathbbold}{bbold}

\newcommand{\map}[3]{#1: #2 \rightarrow #3}

\renewcommand{\theenumi}{(\roman{enumi})}

\newcommand{\range}{\mathop{\bf range}}
\newcommand{\subspace}{\mathop{\bf sub}}
\newcommand{\nullspace}{\mathop{\bf null}}
\DeclareMathOperator{\col}{col}
\DeclareMathOperator{\subto}{subject~to}

\DeclareMathOperator{\mathspan}{span}

\newcommand{\define}{\coloneqq}
\newcommand\xqed[1]{%
  \leavevmode\unskip\penalty9999 \hbox{}\nobreak\hfill
  \quad\hbox{#1}}
\newcommand\envend{\hfill \xqed{$\triangle$}}

\newcommand\oprocendsymbol{\hbox{$\square$}}
\newcommand\oprocend{\relax\ifmmode\else\unskip\hfill\fi\oprocendsymbol}

\newtheorem{theorem}{Theorem}[section]
\newtheorem{lemma}[theorem]{Lemma}
\newtheorem{definition}[theorem]{Definition}

\newtheorem{proposition}[theorem]{Proposition}
\newtheorem{problem}[theorem]{Problem}
\newtheorem{remark}[theorem]{Remark}
\newtheorem{assumption}[theorem]{Assumption}

\newenvironment{pfof}[1]{\vspace{1ex}\noindent{\itshape Proof of
    #1:}\hspace{0.5em}} {\hfill\oprocend\vspace{1ex}}
\newenvironment{proof}[1]{\vspace{1ex}\noindent{\itshape Proof:}\hspace{0.5em}} {\hfill\oprocend\vspace{1ex}}

\newcommand{\rank}{\mathop{\bf rank}}

\usepackage{makecell}

\begin{document}
\title{Linear-Convex Optimal Steady-State Control}

\author{Liam~S.~P.~Lawrence~\IEEEmembership{Student Member,~IEEE}, John~W.~Simpson-Porco,~\IEEEmembership{Member,~IEEE}, and Enrique Mallada~\IEEEmembership{Member,~IEEE}\\
  \thanks{L.~S.~P.~Lawrence and J.~W.~Simpson-Porco are with the Department of Electrical and Computer Engineering, University of Waterloo, Waterloo ON, N2L 3G1 Canada. Email:{\tt \{lsplawre,jwsimpson\}@uwaterloo.ca}.

    E. Mallada is with the Department of Electrical and Computer Engineering, The Johns Hopkins University, 3400 N. Charles Street, Baltimore, MD 21218, Email: {\tt mallada@jhu.edu}}
  \thanks{We acknowledge the support of the Natural Sciences and Engineering Research Council of Canada (NSERC), [CGS M Program, Discovery Grant RGPIN-2017-04008].

    Cette recherche a \'{e}t\'{e} financ\'{e}e par le Conseil de recherches en sciences naturelles et en g\'{e}nie du Canada (CRSNG), [Le Programme de BESC M, Le Programme de subventions \`{a} la d\'{e}couverte RGPIN-2017-04008].

    This work was also supported in part by UW ECE start-up funding.}
}

%----Headers-----
\markboth{Submitted to IEEE Transactions on Automatic Control. This version: \today}%
{Submitted to IEEE Transactions on Automatic Control. This version: \today}

\maketitle

\begin{abstract}
We consider the problem of designing a feedback controller for a multivariable linear time-invariant system which regulates an arbitrary system output to the solution of an equality-constrained convex optimization problem despite unknown constant exogenous disturbances; we term this the linear-convex optimal steady-state (OSS) control problem.
We introduce the notion of an optimality model, and show that the existence of an optimality model is sufficient to reduce the OSS control problem to a stabilization problem.
This yields a constructive design framework for optimal steady-state control that unifies and extends existing design methods in the literature. We illustrate the approach via an application to optimal frequency control of power networks, where our methodology recovers centralized and distributed controllers reported in the recent literature.
\end{abstract}


% -------------------------------------------------------------------------------------- %
\section{Introduction}\label{Sec:Introduction}
Many engineering systems are required to operate at an ``optimal'' steady-state defined by the solution of a constrained optimization problem that seeks to minimize operational costs while satisfying operational constraints. Consider, for example, the problem of optimizing the production setpoints of generators in an electric power system while maintaining supply-demand balance and system stability. The current approach involves a time-scale separation between the optimization and control objectives: optimal generation setpoints are computed offline using demand projections and a model of the network, then the operating points are dispatched as reference commands to local controllers at each generation site \cite{AKB-AA-TS:14}. This process is repeated with a fixed update rate: a new optimizer is computed, dispatched, and tracked. If the model is precise and the supply and demand of power change on a time scale that is slow compared to the update rate, then this method is adequate.

If however the model is poor, and the optimizer changes rapidly (as is the case for power networks with a high penetration of renewable energy sources) the conventional approach can be inefficient \cite{TS-CDP-AvdS:17}; costs are increased as a result of operating sub-optimally. It would then be advantageous to (i) enhance robustness by incorporating feedback, and (ii) eliminate the time-scale separation, by combining the local controllers with a feedback-based online optimization algorithm, so that the optimal operating condition could be tracked in real time; see \cite{AJ:07,JWSP-FD-FB:13t,JWSP-FD-FB:12u,JWSP-BKP-NM-FD:16c,FD-SG:17,EM-CZ-SL:17,EM:16,ED-AS:16,AH-SB-GH-FD:16,NL-CZ-LC:16}.for power system control approaches in this direction.

The same theme of real-time regulation of system variables to optimal values emerges in diverse areas. Fields of application besides the power network control example mentioned already include network congestion management \cite{EM-FP:08,SHL-FP-JCD:02}, chemical processing \cite{MG-DD-MP:05}, wind turbine power capture \cite{BB-HS:10}, and temperature regulation in energy-efficient buildings \cite{TH-XZ-WS-MZ-NL:17}. The breadth of applications motivates the need for a \emph{general} theory and design procedure for controllers that regulate a plant to a maximally efficient operating point defined by an optimization problem, even as the optimizer changes over time due to changing market prices, disturbances to the plant dynamics, and operating constraints that depend on exogenous variables. We refer to the problem of designing such a controller as the \emph{optimal steady-state (OSS) control problem}.

A number of recent publications have formulated problem statements and solutions for variants of the OSS control problem \cite{LSPL-ZEN-EM-JWSP:18e,AJ-ML-PPJvdB:09b,EDA-SVD-GBG:15,ZEN-EM:18,XZ-AP-NL:18,SM-CE:16,SM-AH-SB-GH-FD:18,MC-EDA-AB:18}. Many of the currently-proposed controllers, however, have limited applicability: some solutions only apply to systems of a special form \cite{XZ-AP-NL:18}; some require asymptotic stability of the uncontrolled plant \cite{MC-EDA-AB:18,SM-AH-SB-GH-FD:18}; some attempt to optimize only the steady-state input \cite{EDA-SVD-GBG:15} or output \cite{AJ-ML-PPJvdB:09b,KH-JPH-KU:14,FDB-HBD-CE:12,SM-CE:16} alone; some apply only to equality-constrained \cite{MC-EDA-AB:18} or unconstrained optimization problems \cite{LH-MH-NVDW:18}.

Broadly speaking, these design methodologies consist of modifying an off-the-shelf optimization algorithm to accept system measurements; the algorithm then produces a converging estimate of the optimal steady-state control input, yielding a feedback controller. This procedure, while modular, unnecessarily restricts the design space of dynamic controllers. Moreover, none of the reported approaches adequately considers the impact of the system model on the achievable optimal operating points. Our goal in this paper is to present a framework which widens the design space of optimal steady-state controllers while fully incorporating the system model into the steady-state optimization problem.

\subsection{Contributions}
We consider the \emph{linear-convex OSS control problem}, in which the plant is a finite-dimensional linear time-invariant (LTI) state-space system, the optimization problem has a convex cost function and affine equality constraints, and the disturbances are constant in time. This paper has three contributions. First, we introduce the notion of an optimality model, a dynamic filter which reduces the OSS control problem to a stabilization problem, and provide several explicit constructions of optimality models. Second, we prove that for any of our optimality models, the existence of a stabilizing controller is guaranteed under mild assumptions when the objective function of the optimization problem is quadratic. Third and finally, we apply our results to a problem from power systems control, and show that our general methodology is flexible enough to recover centralized and distributed controllers from the recent literature.

\subsection{Notation}
%The set of $n \times m$ matrices with real entries is denoted by $\real^{n \times m}$. 
The symbol $\bullet$ in $\real^{\bullet \times \bullet}$ indicates that the dimension is unspecified. For a class $C^1$ map $\map{f}{\real^n}{\real}$, $\map{\nabla f}{\real^n}{\real^n}$ denotes its gradient. When the arguments of a function $\map{f}{\real^n \times \real^m}{\real}$ are separated by a semicolon, $\nabla f(x;y)$ refers to the gradient of $f$ with respect to its first argument, evaluated at $(x,y)$. The symbol $\vzeros[]$ denotes a matrix or vector of zeros whose dimensions can be inferred from context. The symbol $\vones[n]$ denotes the $n$-vector of all ones. For scalars or column vectors $\{v_1,v_2,\ldots,v_k\}$, $\col(v_1,v_2,\ldots,v_k)$ is a column vector obtained by vertical concatenation of $v_1,\ldots,v_k$. For vectors $\alpha$ and $\beta$, the notation $\alpha \geq \beta$ indicates that every entry of $\alpha$ is greater than or equal to the corresponding entry of $\beta$. For symmetric matrices $A$ and $B$, $A \succ B$ means $A-B$ is positive definite, while $A \succeq B$ means $A-B$ is positive semidefinite.

\section{Problem Statement}\label{Sec:Problem-Statement}
In the linear-convex optimal steady-state control problem, our objective is to design a feedback controller for a linear time-invariant plant so that a specified output is asymptotically driven to a cost-minimizing steady-state, determined by the solution of a convex optimization problem.
In contrast to a standard static optimization problems, we must contend with closed-loop stability in addition to optimizing a set of decision variables. The plant is a linear time-invariant system subject to an unknown constant disturbance $w \in \real^{n_w}$

\begin{equation}\label{Eq:LTI-Plant}
  \begin{aligned}
    \dot{x} &= Ax + Bu + B_{w}w\,, \qquad x(0) \in \real^n\,,\\
    y &= Cx + Du + Qw\,,\\
    y_{\rm m} &= h_{\rm m}(x,u,w).
  \end{aligned}
\end{equation}

For reasons that will become clear (see Assumption \ref{Ass:Measure}), the measurements $y_{\rm m}$ are permitted to be general nonlinear functions of state, input, and disturbance. The vector $y \in \real^p$ is the \emph{optimization output}, containing states, tracking errors, and control inputs that should be driven to cost-minimizing values in equilibrium.

We will explicitly enforce that the optimization of $y$ is consistent with steady-state operation of the plant. Let $\overline{Y}(w)$ be the set of optimization outputs achievable from a forced equilibrium of \eqref{Eq:LTI-Plant}:
\begin{equation}\label{Eq:Overline-Y}
  \begin{aligned}
    \overline{Y}(w) \define \left\lbrace \bar{y} \in \real^p\;\right|&\; \text{there exists an $(\bar{x},\bar{u})$ such that}\\
    &\,\vzeros[] = A\bar{x}+B\bar{u}+B_ww\\
    &\left.\bar{y} = C\bar{x}+D\bar{u}+Qw\right\rbrace.
  \end{aligned}
\end{equation}
We rewrite $\overline{Y}(w)$ in algebraic form so that we may include membership in $\overline{Y}(w)$ as a constraint of the optimization problem in standard equality form. For each $w$, the set $\overline{Y}(w)$ is an affine subspace of $\real^p$. It may therefore be written as the sum of a (non-unique) ``offset vector'' and a unique subspace, which we denote by $\subspace(\overline{Y}(w))$. In the following lemma, we construct a matrix $G$ whose columns span this unique subspace.

\smallskip

\begin{lemma}[\bf Construction of $\boldsymbol{G}$]\label{Lem:Construction-Of-V}
  Fix a $\tilde{y}(w) \in \overline{Y}(w)$. If $\mathcal{N} \in \real^{(n+m)\times \bullet}$ is a matrix such that 
$$
\range \mathcal{N} = \nullspace\begin{bmatrix}A&B\end{bmatrix},
$$
then the columns of the matrix 
  \begin{equation}\label{Eq:Gmatrix}
  G \define \begin{bmatrix}C&D\end{bmatrix}\mathcal{N} \in \real^{p \times \bullet}
  \end{equation}
  span the subspace $\subspace(\overline{Y}(w))$.\envend
\end{lemma}

\smallskip

%\begin{remark}[\bf $\boldsymbol{G}$ is a Generalized DC Gain]\label{Rem:Gen-DC-Gain}

The proof is straightforward and is omitted. Note that when $A$ is invertible, one may select 
  $$
  \mathcal{N} \define \begin{bmatrix}-A^{-1}B\\I_m\end{bmatrix}
  $$
which yields $G = -CA^{-1}B+D$. This is precisely the DC gain matrix of the $u \to y$ channel for the plant \eqref{Eq:LTI-Plant}. One may think of $G$ in \eqref{Eq:Gmatrix} as a generalization of this, which one can compute regardless of whether or not $A$ is invertible.

From Lemma \ref{Lem:Construction-Of-V} it follows that
\begin{equation}\label{Eq:YEquation}
\bar{y} \in \overline{Y}(w) \,\, \Longleftrightarrow \,\, \text{there exists}\,\,v \in \real^{\bullet}\,\,\text{s.t.}\,\,\bar{y} = \tilde{y}(w) + Gv
\end{equation}
Now let $G_{\perp} \in \real^{\bullet \times p}$ be any full-row-rank matrix satisfying $\nullspace G_{\perp} = \range G$. Then from \eqref{Eq:YEquation}, one finds that
\begin{equation}\label{Eq:Overline-Y-G-b}
  \overline{Y}(w) = \setdef{\bar{y} \in \real^p}{G_{\perp}\bar{y} = b(w)}.
\end{equation}
where $b(w) \define G_{\perp}\tilde{y}(w)$. 
We will see shortly that, for our controller design, the matrix $G_{\perp}$ is important and the vector $b(w)$ is unimportant. 

We can now formulate an optimization problem to determine the desired optimal point for $\bar{y}$ as
\begin{subequations}\label{Eq:Convex-Opt}
  \begin{align}\label{Eq:Convex-Opt-Cost}
    \minimize_{\bar{y} \in \real^p} &\quad f(\bar{y};w)\\
    \subto &\quad G_{\perp}\bar{y} = b(w)\label{Eq:Convex-Opt-Steady-State}\\
                              &\quad H\bar{y} = Lw. \label{Eq:Convex-Opt-Engineering-Equality}
  \end{align}
\end{subequations}
The cost $f$ in \eqref{Eq:Convex-Opt} is our steady-state performance criterion; we assume $f$ is differentiable and convex in $\bar{y}$ for each $w$. The constraint \eqref{Eq:Convex-Opt-Steady-State} is the equilibrium constraint just discussed. The constraints \eqref{Eq:Convex-Opt-Engineering-Equality} represent $n_{\rm ec}$ engineering equality constraints determined by the matrices $H \in \real^{n_{\rm ec} \times p}$ and $L \in \real^{n_{\rm ec} \times n_w}$. We assume that for every $w$, the problem \eqref{Eq:Convex-Opt} has a unique optimizer $\bar{y}^\star$, and a feasible region with non-empty relative interior.

A general nonlinear feedback controller for \eqref{Eq:LTI-Plant} is given by
\begin{equation}\label{Eq:General-Controller}
\begin{aligned}
\dot{x}_{\rm c} &= f_{\rm c}(x_{\rm c},y_{\rm m})\,, \qquad x_{\rm c}(0) \in \real^{n_{\rm c}}\,,\\
u &= h_{\rm c}(x_{\rm c},y_{\rm m}).
\end{aligned}
\end{equation}
The function $f_{\rm c}$ is assumed to be locally Lipschitz in $x_{\rm c}$ and continuous in $y_{\rm m}$, while $h_{\rm c}$ is assumed to be continuous. The dynamics of the closed-loop system consist of \eqref{Eq:LTI-Plant} and \eqref{Eq:General-Controller}. 

Our objective in linear-convex OSS control (for brevity, we will omit ``linear-convex'' in the sequel) is to drive the optimization output $y$ of the plant \eqref{Eq:LTI-Plant} to the solution $\bar{y}^\star(w)$ of the convex optimization problem \eqref{Eq:Convex-Opt} using a feedback controller while ensuring well-posedness and stability of the closed-loop system.
The formal statement is as follows.
For a given $w$, the closed-loop system is said to be \emph{well-posed} if the control input $u$ is uniquely defined for any choice of $(x,x_{\rm c}) \in \real^{n} \times \real^{n_{\rm c}}$, i.e., the equation $u = h_{\rm c}(x_{\rm c},h_{\rm m}(x,u,w))$ is uniquely solvable in $u$.

\smallskip

\begin{problem}[\bf OSS Control]\label{Prob:OSS}
For the plant \eqref{Eq:LTI-Plant}, design, if possible, a dynamic feedback controller of the form \eqref{Eq:General-Controller} such that for every $w$:
\begin{enumerate}[(i)]
\item the closed-loop system is well-posed;\label{Enum:Well-Posedness}
\item the closed-loop system possesses a globally asymptotically stable equilibrium point;\label{Enum:Uniform-Convergence}
\item for every initial condition of the closed-loop system, $\lim\limits_{t \to \infty}y(t) = \bar{y}^\star(w)$. \label{Enum:Error-Zeroing}
\envend
\end{enumerate}
\end{problem}

\smallskip

\begin{remark}[\bf Constant Disturbances]\label{Rem:Disturbances}
We assume throughout that the unmeasured disturbances $w$ are constant, which will lead us to incorporate integral action into our controllers; this is by far the most important case in practice. In reality of course, disturbances (and hence, the optimal operating point) will change in a non-stepwise fashion over time, and the quality of tracking will depend on the rate of variation of the disturbance and on the closed-loop bandwidth. For example, if $\dot{w}$ is bounded, the integral-type controllers we develop will track the optimal operating point with bounded error (see, e.g., \cite{MC-EDA-AB:18}). This is acceptable in practice, and we defer more detailed disturbance models to future studies.
%So long as variations in $w$ occur on a much slower time scale than the closed-loop system dynamics, the controllers presented should achieve approximate optimizer tracking with an error acceptable for practical applications.
\envend
\end{remark}

\smallskip

\begin{remark}[\bf Relation to Optimal Control]\label{Rem:Opt-Control}
The OSS control problem appears similar to an infinite-horizon optimal tracking control problem; however, the two are distinct in both their assumptions and demands. In the latter, one minimizes a cost functional over system trajectories leading to a HJB equation; determining the optimal feedback policy is computationally expensive and the policy will require state and disturbance measurements. The OSS control problem is much less demanding; we ask only for optimal behaviour asymptotically, not optimal trajectories. As a result, we encounter no computational bottlenecks, and do not need to assume the full plant state and all disturbances are measurable.\envend
\end{remark}

\smallskip

\begin{remark}[\bf Relation to Extremum-Seeking Control]\label{Rem:Extremum-Seeking}
The OSS control problem is similar to the online optimization problems considered in the extremum-seeking control literature, e.g., \cite{MK-HHW:00,DD-MG:05}. Extremum seeking is a model-free optimization method, which introduces a sinusoidal probing signal into the system to estimate the static relationship between the control inputs and the desired performance index. In contrast, OSS control explicitly incorporates the equilibrium model of the plant into the optimization problem for this same purpose, and as such does not require the introduction of a probing signal. \envend
%In contrast to the approach of \cite{DD-MG:05}, for example, our optimization problem includes the control input $u$, and we jointly optimize over steady-state inputs and outputs. Also, as noted in \cite{DD-MG:05}, many earlier extremum-seeking control approaches, such as \cite{MK-HHW:00}, only guarantee local attraction to the optimizer. Our problem formulation applies to a smaller class of plants (LTI systems), but, as a result, we are able to speak of global attraction to the optimizer.\envend
\end{remark}

Under the assumptions on the optimization problem \eqref{Eq:Convex-Opt}, the \emph{Karush-Kuhn-Tucker (KKT) conditions} are necessary and sufficient for optimality \cite[Sections 5.2.3 and 5.5.3]{SB-LV:04}. For each $w$, the optimal solution $\bar{y}^\star \in \real^p$ is characterized as the unique vector such that $\bar{y}^\star$ is feasible for \eqref{Eq:Convex-Opt} and there exist $\lambda^\star \in \real^r$, $\mu^\star \in \real^{n_{\rm ec}}$ such that $(\bar{y}^\star,\lambda^\star,\mu^\star)$ satisfies the \emph{gradient condition}
\begin{equation}
\vzeros[] = \nabla f(\bar{y}^\star;w) + G_{\perp}^{\sf T}\lambda^\star + H^{\sf T}\mu^\star\,. \label{Eq:KKT-Gradient}
\end{equation}

\section{OSS Controller Design Framework}\label{Sec:Linear-OSS-Solutions}
The main difficulty in solving the OSS control problem is that the optimizer $\bar{y}^\star(w)$ is unknown, and thus the optimality error $y-\bar{y}^\star(w)$ cannot be directly computed. In our design framework, we propose using a dynamic filter called an \emph{optimality model} to convert the OSS control problem to a related \emph{output regulation} problem. One then solves this output regulation problem using an integral controller and a stabilizing controller. An optimality model therefore reduces the OSS control problem to a stabilization problem. For background on output regulation and integral controllers, see \cite{EJD-AG:75} and \cite[Section 12.3]{HKK:02}.

\subsection{Optimality Models and Reduction to Stabilization Problem}
An optimality model is a filter applied to the measured output $y_{\rm m}$ of the plant that produces a signal $\epsilon$ which acts as a \emph{proxy} for the optimality error $y-\bar{y}^\star(w)$. To make this idea precise, consider a filter $(\varphi,h_{\epsilon})$ with state $\xi \in \real^{n_\xi}$, input $y_{\rm m}$, output $\epsilon \in \real^{n_\epsilon}$, and dynamics
\begin{equation}\label{Eq:Filter}
    \dot{\xi} = \varphi(\xi,y_{\rm m})\,, \quad \epsilon = h_{\epsilon}(\xi,y_{\rm m}).
\end{equation}

\smallskip

\begin{definition}[\bf Optimality Model]\label{Def:Optimality-Model}
  The filter \eqref{Eq:Filter} is said to be an \emph{optimality model} (for the OSS control problem, Problem \ref{Prob:OSS}) if the following implication holds: if the triple $(\bar{x},\bar{\xi},\bar{u}) \in \real^n \times \real^{n_\xi} \times \real^m$ satisfies
  \begin{equation}\label{Eq:Regulator-Filtered}
    \begin{aligned}
      \vzeros[] &= A\bar{x}+B\bar{u}+B_ww\\
      \vzeros[] &= \varphi(\bar{\xi},h_{\rm m}(\bar{x},\bar{u},w))\\
      \vzeros[] &= h_{\epsilon}(\bar{\xi},h_{\rm m}(\bar{x},\bar{u},w))
    \end{aligned}
  \end{equation}
  then the pair $(\bar{x},\bar{u}) \in \real^n \times \real^m$ satisfies
\[
    \bar{y}^{\star}(w) = C\bar{x}+D\bar{u}+Qw.
\]
\end{definition}

\smallskip

% An optimality model encodes sufficient conditions for optimality during steady-state operation with the plant. For instance, given a convex optimization problem where strong duality holds, the optimality model might encode the KKT conditions when it is in dynamic steady-state with the plant; we explore this case further in Section \ref{Sec:Optimality-Model-Design}.

In the OSS control framework, the optimality model is cascaded with the plant, and we then  attempt to solve the (constant disturbance) output regulation problem with $\epsilon$ as the (measurable) error signal. This converts the OSS control problem to stabilization of the \emph{augmented plant}
\begin{subequations}\label{Eq:OSS-Augmented-Plant}
  \begin{align}
    \dot{x} &= Ax+Bu+B_ww\,,\label{SubEq:Plant}\\
    \dot{\xi} &= \varphi(\xi,h_{\rm m}(x,u,w))\,,\label{SubEq:Filter}\\
    \dot{\eta} &= \epsilon \define h_{\epsilon}(\xi,h_{\rm m}(x,u,w))\label{SubEq:Integrator}
  \end{align}
\end{subequations}
using a \emph{stabilizer}
\begin{subequations}\label{Eq:Stabilizer}
  \begin{align}
    \dot{x}_{\rm s} &= f_{\rm s}(x_{\rm s},\eta,\xi,y_{\rm m},\epsilon)\,,\label{SubEq:Stabilizer}\\
    u &= h_{\rm s}(x_{\rm s},\eta,\xi,y_{\rm m},\epsilon).\label{SubEq:Control-Input}
  \end{align}
\end{subequations}
This design framework (Figure \ref{Fig:OSS}) is justified by the following theorem, a proof of which may be found in the appendix.

	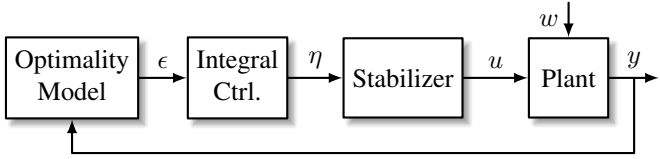
\begin{figure}[ht!]
	\begin{center}  
    \begin{tikzpicture}[auto, scale = 0.6, node distance=2cm,>=latex', every node/.style={scale=1}]
      % \tikzstyle{block} = [draw, fill=none, rectangle, 
      % minimum height=3em, minimum width=6em]
      % \tikzstyle{sum} = [draw, fill=none, circle, node distance=1cm]
      % \tikzstyle{input} = [coordinate]
      % \tikzstyle{output} = [coordinate]
      \tikzstyle{anch} = [coordinate]
      % \tikzstyle{pinstyle} = [pin edge={to-,thin,black}]
      \tikzstyle{block} = [draw, fill=white, rectangle, 
      minimum height=3em, minimum width=6em, blur shadow={shadow blur steps=5}]
      \tikzstyle{smallblock} = [draw, fill=white, rectangle, 
      minimum height=3em, minimum width=5em, blur shadow={shadow blur steps=5}]
      \tikzstyle{hold} = [draw, fill=white, rectangle, 
      minimum height=3em, minimum width=3em, blur shadow={shadow blur steps=5}]
      \tikzstyle{dzblock} = [draw, fill=white, rectangle, minimum height=3em, minimum width=4em, blur shadow={shadow blur steps=5},
      path picture = {
        \draw[thin, black] ([yshift=-0.1cm]path picture bounding box.north) -- ([yshift=0.1cm]path picture bounding box.south);
        \draw[thin, black] ([xshift=-0.1cm]path picture bounding box.east) -- ([xshift=0.1cm]path picture bounding box.west);
        \draw[very thick, black] ([xshift=-0.5cm]path picture bounding box.east) -- ([xshift=0.5cm]path picture bounding box.west);
        \draw[very thick, black] ([xshift=-0.5cm]path picture bounding box.east) -- ([xshift=-0.1cm, yshift=+0.4cm]path picture bounding box.east);
        \draw[very thick, black] ([xshift=+0.5cm]path picture bounding box.west) -- ([xshift=+0.1cm, yshift=-0.4cm]path picture bounding box.west);
      }]
      \tikzstyle{sum} = [draw, fill=white, circle, node distance=1cm, blur shadow={shadow blur steps=8}]
      \tikzstyle{input} = [coordinate]
      \tikzstyle{output} = [coordinate]
      \tikzstyle{split} = [coordinate]
      \tikzstyle{pinstyle} = [pin edge={to-,thin,black}]
      % We start by placing the blocks
%      \node [input,name=input] {};
%      \node [sum, right of=input,node distance = 1cm] (sum) {};
      \node [smallblock] (optmodel) {\makecell[c]{Optimality\\Model}};
      \node [anch, below of=optmodel, node distance=1cm] (anch) {};
      \node [hold, right of=optmodel, node distance=2.2cm] (integral) {\makecell[c]{Integral\\Ctrl.}};
      \node [hold, right of=integral, node distance=2.2cm] (stabilizer) {Stabilizer};
      \node [hold, right of=stabilizer,
      node distance=2.2cm] (system) {Plant};
      \node [input,name=disturbance, above of = system, node distance=1cm] {};
      \node [output,right of = system, node distance=1.2cm] (output) {};
      % We draw an edge between the controller and system block to 
      % calculate the coordinate u. We need it to place the measurement block. 
      \draw [thick, -latex] (optmodel) -- node[name=eps] {$\epsilon$} (integral);
      \draw [thick, -latex] (integral) -- node[name=eps] {$\eta$} (stabilizer);
      \draw [thick, -latex] (stabilizer) -- node[name=u] {$u$} (system);
      % Once the nodes are placed, connecting them is easy. 
%      \draw [thick, -latex] (input) -- node {$r$} (sum);
      \draw [thick, -latex] (system) -- node[name=y] {$y$} (output);
      \draw [thick, -latex] (y) |- (anch) -| (optmodel);
      \draw [thick, -latex] (disturbance) -- node[left] {$w$} (system.north);
    \end{tikzpicture}
    \medskip
	\end{center}
	\caption{ Block diagram of OSS control architecture.}
	\label{Fig:OSS}	
	\end{figure}

\smallskip

\begin{theorem}[\bf Reduction of OSS to Stabilization]\label{Thm:Reduction}
Suppose that $(\varphi,h_{\epsilon})$ is an optimality model. If the stabilizer $(f_{\rm s},h_{\rm s})$ is designed such that the closed-loop system \eqref{Eq:OSS-Augmented-Plant}--\eqref{Eq:Stabilizer} is well-posed and possesses a globally asymptotically stable equilibrium point for every $w$, then the controller \eqref{SubEq:Filter}, \eqref{SubEq:Integrator}, \eqref{SubEq:Stabilizer}, \eqref{SubEq:Control-Input} solves the OSS control problem.\envend
\end{theorem}

\smallskip

Solving the OSS control problem therefore amounts to (i) designing an optimality model and (ii) designing (if possible) a stabilizer for the augmented plant.

\subsection{Optimality Model Design}\label{Sec:Optimality-Model-Design}
\subsubsection{The Gradient Condition}\label{Sec:Gradient-Condition}
According to Definition \ref{Def:Optimality-Model}, an optimality model encodes sufficient conditions for optimality when it is in equilibrium with the plant and its output $\epsilon$ is held at zero. We can incorporate the KKT conditions --- which are sufficient for optimality under our assumptions --- into an optimality model for this purpose. 

Note that the gradient condition \eqref{Eq:KKT-Gradient} involves the dual variable $\lambda^\star$ associated with the equilibrium constraints. Dual variables associated with equality constraints are typically calculated using an integrator on the constraint violation: see \cite[Equation (4a)]{AJ-ML-PPJvdB:09b}, or \cite[Equation (8f)]{XZ-AP-NL:18}, for example. Unfortunately integrating the equilibrium constraint violation $G_{\perp}y-b(w)$ is impossible, since $w$ is unknown. Luckily, doing so is also unnecessary, since the constraint $G_\perp y = b(w)$ is satisfied at any forced equilibrium point of the physical plant, by the definition of $G_{\perp}$ and $b(w)$; recall \eqref{Eq:Overline-Y} and \eqref{Eq:Overline-Y-G-b}. We now describe how to incorporate the gradient condition \eqref{Eq:KKT-Gradient} into an optimality model without calculating $\lambda^\star$. 

Let $G$ be the matrix of Lemma \ref{Lem:Construction-Of-V}. Recall that $\nullspace G_{\perp} = \range G$; by taking the orthogonal complement of both sides, it follows that $\range G_{\perp}^{\sf T} = \nullspace G^{\sf T}$. Rearranging \eqref{Eq:KKT-Gradient} to read $G_\perp^{\sf T}(-\lambda^\star) = \nabla f(\bar{y}^\star;w)+H^{\sf T}\mu^\star$, 
we see that $\nabla f(\bar{y}^\star;w)+H^{\sf T}\mu^\star \in \range G_\perp^{\sf T} = \nullspace G^{\sf T}$. That is, the existence of a triple $(\bar{y}^\star,\lambda^\star,\mu^\star)$ satisfying \eqref{Eq:KKT-Gradient} is equivalent to the existence of a pair $(\bar{y}^\star,\mu^\star)$ satisfying
\begin{equation}\label{Eq:GDeltaeps}
  G^{\sf T}\left(\nabla f(\bar{y}^\star;w) + H^{\sf T}\mu^\star \right) = \vzeros[].
\end{equation}
The left-hand side of this equation is a natural choice for inclusion in the proxy error signal $\epsilon$, since driving $\epsilon$ to zero will then enforce the gradient KKT condition. 

Note however that there are other equivalent ways we could rewrite the gradient condition. Define a matrix $T$ such that
\begin{equation}\label{Eq:Def-Of-T}
  \range T = \nullspace \begin{bmatrix}G_{\perp}\\H\end{bmatrix}\,.
\end{equation}
We have $\nullspace T^{\sf T} = \left(\range T\right)^\perp$, which means
\begin{equation*}
\nullspace T^{\sf T}  = \left(\nullspace \begin{bmatrix}G_{\perp}\\H\end{bmatrix}\right)^\perp = \range \begin{bmatrix}G_\perp^{\sf T} & H^{\sf T}\end{bmatrix}\,.
\end{equation*}
Rearrange \eqref{Eq:KKT-Gradient} to read $G_\perp^{\sf T}(-\lambda^\star) + H^{\sf T}(-\mu^\star) = \nabla f(\bar{y}^\star,w)$, and we see that $\nabla f(\bar{y}^\star,w) \in \range \begin{bmatrix}G_\perp^{\sf T} & H^{\sf T}\end{bmatrix} = \nullspace T^{\sf T}$. Therefore, the existence of a triple $(\bar{y}^\star,\lambda^\star,\mu^\star)$ satisfying \eqref{Eq:KKT-Gradient} is equivalent to the existence of a $\bar{y}^\star$ satisfying
\begin{equation}\label{Eq:TDeltaeps}
  T^{\sf T} \nabla f(\bar{y}^\star;w) = \vzeros[].
\end{equation}
This procedure can also be generalized by including only some rows of $H$ in the construction of $T$, leading to a hybrid between \eqref{Eq:GDeltaeps} and \eqref{Eq:TDeltaeps}; the details are omitted. As we did with \eqref{Eq:GDeltaeps}, we can make the expression on the left-hand side of \eqref{Eq:TDeltaeps} one of the components of an optimality model's error output.

\smallskip

\subsubsection{Optimality Models}\label{Sec:Optimality-Models}
We are now ready to construct optimality models for OSS control. We shall require our measurement vector to contain some key information about the optimization problem \eqref{Eq:Convex-Opt}.

\smallskip

\begin{assumption}[\bf Measurement Assumptions]\label{Ass:Measure}
  The measurement vector $y_{\rm m}$ contains the gradient $\nabla f(y,w)$ and the engineering equilibrium constraint violations, $Hy-Lw$.\envend
\end{assumption}

\smallskip

The following three propositions present the \emph{output subspace}, \emph{feasible subspace}, and \emph{reduced-error feasible subspace} optimality models. Proving that these filters are indeed optimality models is done by examining the closed-loop equilibria and showing that the resulting equations are equivalent to the KKT conditions.

\smallskip

\begin{proposition}[\bf Output Subspace Optimality Model (OS-OM)]
\label{Prop:OS-OM}
Let $G$ be the matrix of Lemma \ref{Lem:Construction-Of-V}. The dynamic filter
  \begin{equation}\label{Eq:OS-OM}
    \begin{aligned}
      \dot{\mu} &= Hy-Lw\\
      \epsilon &= G^{\sf T}\left(\nabla f(y;w) + H^{\sf T}\mu\right)
    \end{aligned}
  \end{equation}
  is an optimality model for the OSS control problem.\envend
\end{proposition}

\begin{proof}{}
The proof is similar to the proof of Proposition \ref{Prop:FS-OM}.
\end{proof}

\smallskip

\begin{proposition}[\bf Feasible Subspace Optimality Model (FS-OM)]\label{Prop:FS-OM}
Let $T$ be a matrix satisfying \eqref{Eq:Def-Of-T}. The static filter
  \begin{equation}\label{Eq:FS-OM}
      \epsilon = \begin{bmatrix}Hy-Lw\\T^{\sf T}\nabla f(y;w)\end{bmatrix}
  \end{equation}
  is an optimality model for the OSS control problem.\envend
\end{proposition}
\begin{proof}{}
See the appendix.
\end{proof}

\smallskip

In special circumstances, one can modify the FS-OM above to obtain an optimality model with an error signal of reduced dimension; this reduces the number of integrators required.

\smallskip

\begin{proposition}[\bf Reduced-Error FS-OM (REFS-OM)]\label{Prop:REFS-OM}
Let $G$ be the matrix of Lemma \ref{Lem:Construction-Of-V} and let $T$ be a matrix satisfying \eqref{Eq:Def-Of-T}. Then the static filter
  \begin{equation}\label{Eq:REFS-OM}
    \begin{aligned}
      \epsilon &= Hy-Lw+T^{\sf T} \nabla f(y;w)
    \end{aligned}
  \end{equation}
  is an optimality model for the OSS control problem if $\range HG \cap \range T^{\sf T} = \{\vzeros[]\}$.\envend
%  \begin{equation}\label{Eq:Nonintersecting-Range}
%     ~\text{for all $$.}
%  \end{equation}
\end{proposition}
\begin{proof}{}
See the appendix.
\end{proof}

\smallskip
%
%\begin{remark}[\bf Uncertain Equality Constraints]
%It is possible to apply the RFS optimality models even when the engineering equality constraint matrices are uncertain --- that is, when \eqref{Eq:Convex-Opt-Engineering-Equality} reads $Hy=Lw$ --- by redefining the RFS property as the existence of a matrix $T$ satisfying
%\begin{equation}\label{eq:extended-rfs}
%\range T = \nullspace \begin{bmatrix}G_\perp\\H\end{bmatrix} \forall \delta \in \boldsymbol{\delta}\,.
%\end{equation}
%Under Assumption \ref{Ass:Measure} and using a $T$ satisfying \eqref{eq:extended-rfs}, Propositions \ref{Prop:FS-OM} and \ref{Prop:REFS-OM} still hold with $Hy-Lw$ replaced by $Hy-Lw$.\envend
%\end{remark}

%The intersection condition of Proposition \ref{Prop:REFS-OM} is equivalent to the rank condition
%  \begin{equation*}
%    \rank\begin{bmatrix}HG&T^{\sf T}\end{bmatrix} = \rank HG + \rank T^{\sf T}
%  \end{equation*}
%which may be easier to check.

%\begin{remark}[\bf Range Condition]
%  The property \eqref{Eq:Nonintersecting-Range} is equivalent to
%
%  for all $\delta \in \boldsymbol{\delta}$. Therefore, one may check the rank condition above to determine if \eqref{Eq:Nonintersecting-Range} holds.\envend
%\end{remark}

\smallskip

\subsection{Quadratic Program OSS Control}

We now consider the specific case when the optimization problem \eqref{Eq:Convex-Opt} is an equality-constrained convex quadratic program (QP). We term this variant of the problem QP-OSS control. Under this assumption, the closed-loop system becomes LTI, and we can obtain very explicit results on the existence of a stabilizer (Figure \ref{Fig:OSS}). Suppose the optimization problem \eqref{Eq:Convex-Opt} is of the form
\begin{equation}\label{Eq:Opt-QP-OSS}
  \begin{aligned}
    \minimize_{\bar{y} \in \real^p} &\quad \tfrac{1}{2}\bar{y}^{\sf T}\bar{M}\bar{y} - \bar{y}^{\sf T}Nw\\
    \subto &\quad G_{\perp}\bar{y} = b(w)\\
    &\quad H\bar{y} = Lw,
  \end{aligned}
\end{equation}
where $M \succeq \vzeros[]$.\footnote{Any constant term of the form $\bar{y}^{\sf T}c$ with $c \in \real^p$ may be included in the term $\bar{y}^{\sf T}Nw$ by appropriate redefinition of $N$ and $w$.} Assumption \ref{Ass:Measure} in the present context implies that we can take the available measurements $y_{\rm m}$ as a linear function of $(x,u,w)$, i.e., $y_{\rm m} = C_{\rm m}x+D_{\rm m}u+Q_{\rm m}w$.

Under mild assumptions below, we can ensure that the augmented plant \eqref{Eq:OSS-Augmented-Plant} arising from the FS-OM, OS-OM, or REFS-OM is both stabilizable and detectable, which in turn guarantees that a solution of the OSS control problem exists and can be found using standard LTI design methods.

\smallskip 

\begin{theorem}[\bf Solvability of QP-OSS Control]\label{Thm:QP-OSS}
The QP-OSS control problem is solvable when
\begin{enumerate}
\item $(C_{\rm m},A,B)$ is stabilizable and detectable,\label{cond:cab-stab}
\item a unique primal solution to \eqref{Eq:Opt-QP-OSS} exists, and \label{cond:primal-unique}
\item at least one of the following holds:
\begin{enumerate}
\item a unique dual solution to \eqref{Eq:Opt-QP-OSS} exists; \label{cond:dual-unique}
\item $\range HG \cap \range T^{\sf T} = \{\vzeros[]\}$ and $(\range HG)^\perp \cap (\range T^{\sf T})^\perp = \{\vzeros[]\}$.\label{cond:subspaces} \envend
\end{enumerate}
\end{enumerate}
\end{theorem}
\begin{proof}{}
When (i), (ii), and (iii)(a) hold, one can show that the augmented plant arising from the FS-OM or OS-OM may be made stabilizable and detectable. When (i), (ii), and (iii)(b) hold, one can show that the augmented plant arising from the REFS-OM is stabilizable and detectable. See the appendix for details.
\end{proof}

\smallskip

\section{Case Study: Optimal Frequency Regulation in Power Systems}\label{Sec:Case-Study}
This final section illustrates the application of our theory to a power system control problem. Our main objective is to work through the constructions presented in Section \ref{Sec:Linear-OSS-Solutions}, and to simultaneously illustrate the many sources of design flexibility within our proposed framework. In particular, we will show that centralized and distributed frequency controllers proposed in the literature are recoverable as special cases of our framework.

The dynamics of synchronous generators in a connected AC power network with $n$ buses and $n_{\rm t}$ transmission lines is modelled in a reduced-network framework by the \emph{swing equations}. The vectors of angular frequency (deviations from nominal) $\omega \in \real^{n}$ and real power flows $p \in \real^{n_{\rm t}}$ along the transmission lines obey the dynamic equations
\begin{equation}\label{Eq:Swing}
  \begin{aligned}
    M\dot{\omega} &= P^\star-D\omega - \mathcal{A}p + u\\
    \dot{p} &= \mathcal{B}\mathcal{A}^{\sf T}\omega,
  \end{aligned}
\end{equation}
in which $M \succ \vzeros[]$ is the (diagonal) inertia matrix, $D \succ \vzeros[]$ is the (diagonal) damping matrix, $\mathcal{A} \in \{0,1,-1\}^{n \times n_{\rm t}}$ is the signed node-edge incidence matrix of the network, $\mathcal{B} \succ \vzeros[]$ is the diagonal matrix of transmission line susceptances, $P^\star \in \real^n$ is the vector of uncontrolled power injections (generation minus demand) at the buses, and $u \in \real^n$ is the controllable reserve power produced by the generators. The incidence matrix satisfies $\nullspace \mathcal{A}^{\sf T} = \mathspan(\vones[n])$, and strictly for simplicity we assume that the network is acyclic, in which case $n_{\rm t} = n-1$ and $\nullspace \mathcal{A} = \{\vzeros[]\}$. We refer to \cite[Section VII]{CZ-UT-NL-SHL:13} for a first-principles derivation of this model, and remark that our calculations to follow extend without issue to more complex models which include turbine-governor dynamics.

We consider the \emph{optimal frequency regulation problem} (OFRP), wherein we minimize the total cost $\sum_{i} J_i(\bar{u}_i)$ of steady-state reserve power production in the system subject to system equilibrium and zero steady-state frequency deviations:
\begin{equation}\label{Eq:OFRP}
  \begin{aligned}
    \minimize_{\bar{u} \in \real^n, \bar{\omega} \in \real^n} &\quad J(\bar{u}) \define \sum_{i=1}^n \nolimits J_i(\bar{u}_i)\\
    \subto &\quad G_{\perp}\col(\bar{u},\bar{\omega}) = b(w)\\
    &\quad F\bar{\omega} = \vzeros[].
  \end{aligned}
\end{equation}
We shall compute the matrix $G_{\perp}$ of the equilibrium constraints shortly; the vector $b(w)$ is unimportant for controller design. The matrix $F$ encodes the steady-state frequency constraint. We will specify the requirements on $F$ later in this section.

With state vector $x \define \col(\omega,p)$, the dynamics \eqref{Eq:Swing} can be put into the standard LTI form \eqref{Eq:LTI-Plant} with matrices
\begin{equation*}
  \begin{aligned}
    A &\define \begin{bmatrix}-M^{-1}D&-M^{-1}\mathcal{A}\\\mathcal{B}\mathcal{A}^{\sf T}&\vzeros[]\end{bmatrix}, \,\,
    B = B_w \define \begin{bmatrix}M^{-1}\\\vzeros[]\end{bmatrix}.
  \end{aligned}
\end{equation*}
We select the optimization output as $y \define \col(u,\omega)$, so that
\begin{equation}\label{Eq:OFRP-CD}
  C \define \begin{bmatrix}\vzeros[]&\vzeros[]\\I_n&\vzeros[]\end{bmatrix}\quad D \define \begin{bmatrix}I_n\\\vzeros[]\end{bmatrix},
\end{equation}
and we take the measured output as $y_{\rm m} = \col(u,F\omega)$.

We will demonstrate the use of the feasible subspace and reduced-error feasible subspace optimality models of Propositions \ref{Prop:FS-OM} and \ref{Prop:REFS-OM}. We begin by constructing the matrix $G$ of Lemma \ref{Lem:Construction-Of-V} and a matrix $T$ satisfying \eqref{Eq:Def-Of-T}. We first construct a matrix $\mathcal{N}$ satisfying $\range \mathcal{N} = \nullspace \begin{bmatrix}A&B\end{bmatrix}$. One may verify that choosing
\begin{equation}\label{Eq:OFRP-NAB}
  \mathcal{N} \define \begin{bmatrix}\vones[n]&\vzeros[]\\\vzeros[]&I_n\\D\vones[n]&\mathcal{A}\end{bmatrix}
\end{equation}
yields the required property. Using \eqref{Eq:OFRP-NAB} and \eqref{Eq:OFRP-CD}, we calculate $G = \begin{bmatrix}C&D\end{bmatrix}\mathcal{N}$ to be
\begin{equation*}
  G = \begin{bmatrix}D\vones[n]&\mathcal{A}\\\vones[n]&\vzeros[]\end{bmatrix}.
\end{equation*}

Next, we construct a full-row-rank matrix $G_{\perp} \in \real^{n \times 2n}$ satisfying $\nullspace G_{\perp} = \range G$. We find that selecting
\begin{equation*}
  G_{\perp} \define \begin{bmatrix}\vones[n]\vones[n]^{\sf T} & -(\vones[n]^{\sf T}D\vones[n])I_n\end{bmatrix}
\end{equation*}
yields the required property. We identify the matrix $H$ of the engineering equality constraints in \eqref{Eq:Convex-Opt} for the problem \eqref{Eq:OFRP} as $H \define \begin{bmatrix}\vzeros[]&F\end{bmatrix}$. Following \eqref{Eq:Def-Of-T}, we select a matrix $T$ satisfying
\begin{equation}\label{Eq:T0-Case-Study}
  \range T = \nullspace \begin{bmatrix}\vones[n]\vones[n]^{\sf T} & -(\vones[n]^{\sf T}D\vones[n])I_n\\ \vzeros[] & F\end{bmatrix}\,.
\end{equation}
The null space on the right-hand side of \eqref{Eq:T0-Case-Study} is spanned by vectors of the form $\col(v,\vzeros[])$ where $\vones[n]^{\sf T}v = 0$. Inspired by approaches in multi-agent control, we introduce a connected, weighted and directed communication graph $\mathcal{G}_{\rm c} = (\{1,\ldots,n\},\mathcal{E}_{\rm c})$ between the buses, with associated \emph{Laplacian matrix} $L_{\rm c} \in \real^{n \times n}$. We assume the directed graph $\mathcal{G}_{\rm c}$ contains a globally reachable node.\footnote{See \cite[Chapter 6]{FB-LNS} for details.} Under this assumption, we have that $\rank(L_{\rm c}) = n-1$ with $\nullspace(L_{\rm c})$ spanned by $\vones[n]$. It follows that \eqref{Eq:T0-Case-Study} holds with $T = \left[\begin{smallmatrix}L_{\rm c}^{\sf T}\\\vzeros[]\end{smallmatrix}\right]$.
%Therefore, the optimization problem satisfies the robust feasible subspace property.

It further holds that $L_{\rm c}$ has a left null space of dimension one spanned by a nonnegative and non-zero vector $\mathsf{w} \in \real^n$. Assuming that $F$ is selected such that $\mathsf{w}^{\sf T}F\vones[n] \neq \vzeros[]$, the range condition of Proposition \ref{Prop:REFS-OM} is satisfied, and we may apply the REFS-OM \eqref{Eq:REFS-OM} to obtain the optimality model
\begin{equation}\label{Eq:OFRP-REFS-OM}
  \epsilon = F\omega + L_{\rm c}\nabla J(u).
\end{equation}
Therefore, one option for an OSS controller is
\begin{subequations}
  \begin{align}
    \dot{\eta} &= F\omega+L_{\rm c}\nabla J(u)\label{Eq:rerfs-int-1}\\
    u &= -K_{\rm p}\omega - K_{\rm i}\eta,
  \end{align}
\end{subequations}
where $K_{\rm p}, K_{\rm i}$ are gain matrices that should be selected for closed-loop stability/performance. 
%If the objective function $J$ is a positive definite quadratic, one can show that the augmented plant comprising \eqref{Eq:Swing} and \eqref{Eq:rerfs-int-1} is stabilizable and detectable using Proposition \ref{Prop:Stab-Detect-REFS-OM}. 
With $F \define I_n$, $K_{\rm p} = \vzeros[]$ and $K_{\rm i} = \frac{1}{k}I_n$ for $k > 0$,  this design reduces to the \emph{distributed-averaging proportional-integral} (DAPI) frequency control scheme; see \cite{JWSP-FD-FB:12u,FD-JWSP-FB:13y,CZ-EM-FD:15}.

We can obtain several other control schemes by instead applying the FS-OM as our optimality model. Let $F \define \mathsf{c}^{\sf T}$, where $\mathsf{c}$ is a vector of convex combination coefficients satisfying $\mathsf{c}_i \geq 0$ and $\sum_{i=1}^{n}\mathsf{c}_i = 1$. Define $\widetilde{L}_{\rm c} \in \real^{(n-1) \times n}$ as the matrix obtained by eliminating the first row from $L_{\rm c}$ and set $T \define \left[\begin{smallmatrix}\widetilde{L}_{\rm c}^{\sf T}\\\vzeros[]\end{smallmatrix}\right]$.
This choice of $T$ also satisfies \eqref{Eq:T0-Case-Study}. The FS-OM \eqref{Eq:FS-OM} yields the optimality model
\begin{equation}\label{Eq:OFRP-RFSOM-2}
  \epsilon = \begin{bmatrix}\mathsf{c}^{\sf T}\omega\\\widetilde{L}_{\rm c}\nabla J(u)\end{bmatrix}.
\end{equation}
It follows that one option for an OSS controller is
\begin{subequations}\label{Eq:freq-cont-new}
  \begin{align}
    \dot{\eta}_1 &= \mathsf{c}^{\sf T}\omega\label{Eq:int-1}\\
    \dot{\eta}_2 &= \widetilde{L}_{\rm c}\nabla J(u)\label{Eq:int-2}\\
    u &= -K_{\rm p}\omega - K_1\eta_1 - K_2\eta_2.
  \end{align}
\end{subequations}
where again $K_{\rm p}, K_{1}, K_2$ are gain matrices. The interpretation of this (novel) controller is that one agent collects frequency measurements and implements the integral control \eqref{Eq:int-1}, while the other agents average their marginal costs via \eqref{Eq:int-2}. 

If the objective function $J$ is a positive definite quadratic, one can use Theorem \ref{Thm:QP-OSS} to show that a solution to the present OSS control problem is guaranteed to exist. Specifically, for $F \define I_n$, one uses Theorem \ref{Thm:QP-OSS} with condition \ref{cond:subspaces}, and for $F \define \mathsf{c}^{\sf T}$, one uses Theorem \ref{Thm:QP-OSS} with condition \ref{cond:dual-unique}. Moreover, the augmented plant defined by the use of either \eqref{Eq:rerfs-int-1} or \eqref{Eq:int-1}-\eqref{Eq:int-2} can be shown to be stabilizable and detectable using the proof of Theorem \ref{Thm:QP-OSS}.

As a final example, we can recover the ``gather-and-broadcast'' scheme of \cite{FD-SG:17} from the optimality model \eqref{Eq:OFRP-RFSOM-2} as follows. Assume that each $J_i$ is strictly convex, and retain the integral controller \eqref{Eq:int-1}. Next, using the fact that $\nullspace \widetilde{L}_{\rm c} = \mathspan(\vones[n])$, select the input $u$ to zero the second component of $\epsilon$:
$$
\begin{aligned}
\widetilde{L}_{\rm c}\nabla J(u) = \vzeros[] \quad &\Longleftrightarrow &\quad& \nabla J(u) = \alpha \vones[n]\,\,\text{for all}\,\,\alpha \in \real\\
&\Longleftrightarrow &\quad& u = (\nabla J)^{-1}(\alpha\vones[n]),\,\,\text{for all}\,\,\alpha \in \real.
\end{aligned}
$$
Selecting $\alpha = \eta$ leads to the gather-and-broadcast controller
\begin{equation}\label{Eq:DorflerGrammatico}
\dot{\eta} = \sum_{i=1}^{n}\nolimits \mathsf{c}_i\omega_i, \quad u_i = (\nabla J_i)^{-1}(\eta).
\end{equation}

In summary, several recent frequency control schemes, and the novel scheme \eqref{Eq:freq-cont-new}, can be recovered as special cases of our general control framework. The full potential of our methodology for the design of improved power system control will be an area for future study.

\section{Conclusions}
\label{Sec:Conclusions}
% -------------------------------------------------------------------------------------- %
We have studied in detail the linear-convex OSS control problem, wherein we design a controller to guide an LTI system to the solution of an optimization problem despite unknown, constant exogenous disturbances. We introduced the idea of an optimality model, the existence of which allows us to reduce the OSS control problem to a stabilization problem, and presented several candidate filters which under weak conditions are indeed optimality models. The flexibility of the OSS control framework was illustrated through a case study in power system control, showing that its application recovers several existing schemes in the literature.

Future work will present the analogous discrete-time and sampled-data OSS control problems, along with a more detailed study of applications in power system control. A large number of open problems and directions exist, including but not limited to: OSS control for nonlinear systems subject to time-varying disturbances, flexibility of the framework for distributed/decentralized control, formulations and solutions of hierarchical, competitive, and approximate OSS control problems, and the application of the OSS control framework to the design of new optimization algorithms.

\renewcommand{\baselinestretch}{1}
\bibliographystyle{IEEEtran}

\bibliography{alias,JWSP,New,Main}

\appendix

\smallskip

\begin{pfof}{Theorem \ref{Thm:Reduction}} By assumption, the closed-loop system \eqref{Eq:OSS-Augmented-Plant}--\eqref{Eq:Stabilizer} is well-posed and possesses a globally asymptotically stable equilibrium point for each $w$; hence, the first two requirements of Problem \ref{Prob:OSS} are satisfied. It remains to show that $\lim_{t \to \infty}y(t) = y^\star(w)$ for each $w$ and every initial condition. Since the closed-loop system possesses a globally asymptotically stable equilibrium point for each $w$, there exists a unique solution $(\bar{x},\bar{\xi},\bar{\eta},\bar{x}_{\rm s})$ to the steady-state equations
  \begin{equation*}
    \begin{aligned}
      \vzeros[] &= A\bar{x}+B\bar{u}+B_ww\\
      \bar{y}_{\rm m} &= h_{\rm m}(\bar{x},\bar{u},w)
    \end{aligned} \qquad \begin{aligned}
          \vzeros[] &= \varphi(\bar{\xi},\bar{y}_{\rm m})\\
      \vzeros[] &= h_{\epsilon}(\bar{\xi},\bar{y}_{\rm m})
    \end{aligned}\\
  \end{equation*}
  \begin{equation*}
  \begin{aligned}
      \vzeros[] &= f_{\rm s}(\bar{x}_{\rm s},\bar{\eta},\bar{\xi},\bar{y}_{\rm m},\vzeros[])\\
      \bar{u} &= h_{\rm s}(\bar{x}_{\rm s},\bar{\eta},\bar{\xi},\bar{y}_{\rm m},\vzeros[])
      \end{aligned}
  \end{equation*}
  for each $w$. Since $(\varphi,h_{\epsilon})$ is an optimality model, the pair $(\bar{x},\bar{u})$ satisfies $\bar{y}^\star(w) = C\bar{x}+D\bar{u}+Qw$. Because this equilibrium point attracts all trajectories of the closed-loop system and $y(t)$ is continuous, it must be the case that $\lim_{t \to \infty}y(t) = y^\star(w)$ for every $w$ and every initial condition. Therefore, the controller \eqref{SubEq:Filter}, \eqref{SubEq:Integrator}, \eqref{SubEq:Stabilizer}, \eqref{SubEq:Control-Input} solves the OSS control problem.
\end{pfof}

\smallskip

\begin{pfof}{Proposition \ref{Prop:FS-OM}}
  For each $w$, consider the solutions $(\bar{x},\bar{u})$ to:
   \begin{subequations}\label{Eq:RFS-OM-Regulator}
    \begin{align}
      \vzeros[] &= A\bar{x}+B\bar{u}+B_ww \label{SubEq:Plant-Equilibrium}\\
      \bar{y} &= C\bar{x}+D\bar{u}+Qw \label{SubEq:Output-Equilibrium}\\
      \vzeros[] &= H\bar{y}-Lw \label{SubEq:Engineering-Equality}\\
      \vzeros[] &= T^{\sf T}\nabla f_0(\bar{y};w). \label{SubEq:Gradient}
    \end{align}
  \end{subequations}
  The equations \eqref{Eq:RFS-OM-Regulator} correspond to the equations \eqref{Eq:Regulator-Filtered} in the definition of an optimality model. We show that \eqref{Eq:RFS-OM-Regulator} imply the KKT conditions. The first two equations \eqref{SubEq:Plant-Equilibrium} and \eqref{SubEq:Output-Equilibrium} imply $\bar{y} \in \overline{Y}(w)$, which is equivalent to the first set of equality constraints, \eqref{Eq:Convex-Opt-Steady-State}. The equation \eqref{SubEq:Engineering-Equality} is the engineering equality constraint, \eqref{Eq:Convex-Opt-Engineering-Equality}. Finally, because the feasible subspace property holds, \eqref{SubEq:Gradient} implies the gradient condition \eqref{Eq:TDeltaeps}. Since the KKT conditions are sufficient for optimality, the following implication holds for all $w$: if $(\bar{x},\bar{\mu},\bar{u})$ satisfy \eqref{Eq:RFS-OM-Regulator}, then $(\bar{x},\bar{u})$ satisfy
  \begin{equation*}
    \bar{y}^\star(w) = C\bar{x}+D\bar{u}+Qw.
  \end{equation*}
  The filter \eqref{Eq:FS-OM} satisfies the criterion of Definition \ref{Def:Optimality-Model}, and is therefore an optimality model.
\end{pfof}

\smallskip

\begin{pfof}{Proposition \ref{Prop:REFS-OM}}
  For each $w$, consider the solutions $(\bar{x},\bar{u})$ to
  \begin{subequations}\label{Eq:RERFS-OM-Regulator}
    \begin{align}
      \vzeros[] &= A\bar{x}+B\bar{u}+B_ww\label{Eq:RERFS-eq}\\
      \bar{y} &= C\bar{x}+D\bar{u}+Qw\label{Eq:RERFS-out}\\
      \vzeros[] &= H\bar{y}-Lw+T^{\sf T}\nabla f_0(\bar{y};w).\label{Eq:RERFS-opt-2}
    \end{align}
  \end{subequations}
  The equations \eqref{Eq:RERFS-OM-Regulator} correspond to the equations \eqref{Eq:Regulator-Filtered} in the definition of an optimality model. By assumption, the feasible region of the optimization problem \eqref{Eq:Convex-Opt} is non-empty: hence, there exists a $\mathsf{y}(w)$ such that
  \begin{subequations}
    \begin{align}
      G_\perp\mathsf{y}(w) &= b(w)\label{eq:sfy-1}\\
      H\mathsf{y}(w) &= Lw.\label{eq:sfy-2}
    \end{align}
  \end{subequations}
  Equations \eqref{Eq:RERFS-eq} and \eqref{Eq:RERFS-out} imply that $G_\perp\bar{y} = b(w)$. Equation \eqref{eq:sfy-1} and the fact that $\nullspace G_\perp = \range G$ imply there exists a $v$ such that $\bar{y} = \mathsf{y}(w)+Gv$. Substituting this expression for $\bar{y}$ into \eqref{Eq:RERFS-opt-2} and making use of \eqref{eq:sfy-2}, we see that
  \begin{equation}\label{Eq:RERFS-1}
      \vzeros[] = HGv + T^{\sf T}\nabla f_0(\bar{y};w)\,.
  \end{equation}
  Since
  \begin{equation}\label{Eq:RERFS-2}
    \range HG \cap \range T^{\sf T} = \{\vzeros[]\}\,,
  \end{equation}
  \eqref{Eq:RERFS-1} and \eqref{Eq:RERFS-2} imply
  \begin{equation*}
    \begin{aligned}
      \vzeros[] &= HGv\\
      \vzeros[] &= T^{\sf T}\nabla f_0(\bar{y};w)
    \end{aligned}
  \end{equation*}
  for every $w$. Since $H\bar{y}-Lw = HGv$,
  \begin{equation*}
    \begin{aligned}
      \vzeros[] &= H\bar{y}-Lw\\
      \vzeros[] &= T^{\sf T}\nabla f_0(\bar{y};w)
    \end{aligned}
  \end{equation*}
  for every $w$. The remainder of the proof proceeds like the proof of Proposition \ref{Prop:FS-OM}.
\end{pfof}

\smallskip

\begin{pfof}{Theorem \ref{Thm:QP-OSS}} We will apply the classic result \cite[Theorem 1]{EJD-AG:75}, which provides necessary and sufficient conditions for stabilizability and detectability of the augmented plant \eqref{Eq:OSS-Augmented-Plant}. We examine the application of \cite[Theorem 1]{EJD-AG:75} to the augmented plant corresponding to each of the three optimality models --- the FS-OM, the OS-OM, and the REFS-OM --- in sequence. First note that condition \ref{cond:cab-stab} is exactly conditions (a) and (b) of \cite[Theorem 1]{EJD-AG:75}; condition (e) of \cite[Theorem 1]{EJD-AG:75} is automatically satisfied here. We will show that conditions (c) and (d) of \cite[Theorem 1]{EJD-AG:75} are equivalent to conditions \ref{cond:primal-unique} and \ref{cond:dual-unique} in the case of the FS-OM and the OS-OM. In the case of the REFS-OM, the property $\range HG \cap \range T^{\sf T} = \{\vzeros[]\}$ of condition \ref{cond:subspaces} is required for the filter \eqref{Eq:REFS-OM} to be an optimality model by Proposition \ref{Prop:REFS-OM}; we will show that condition \ref{cond:primal-unique} and the second property of \ref{cond:subspaces}, $(\range HG)^\perp \cap (\range T^{\sf T})^\perp = \{\vzeros[]\}$, are equivalent to conditions (c) and (d) of \cite[Theorem 1]{EJD-AG:75}. We first require the following lemmas.

\smallskip

\begin{lemma}[\bf Unique Primal Solution]\label{Lem:Unique-Solution}
    Suppose the optimization problem \eqref{Eq:Opt-QP-OSS} is feasible, and let $T \in \real^{p \times \bullet}$ be any matrix satisfying $\range T = \nullspace \left[\begin{smallmatrix}G_{\perp}\\H\end{smallmatrix}\right]$. Then \eqref{Eq:Opt-QP-OSS} has a unique optimizer if and only if $v^{\sf T}Mv > 0$ on $\range T$.% for all non-zero $v \in \range T$.
      \end{lemma}
  \begin{proof}{}
    Fix a member $\tilde{y}(w)$ of the feasible set of \eqref{Eq:Opt-QP-OSS}. Since $\range T = \nullspace \left[\begin{smallmatrix}G_{\perp}\\H\end{smallmatrix}\right]$, we can rewrite the optimization problem \eqref{Eq:Opt-QP-OSS} as $\minimize_{\bar{y} \in \real^p,\,v \in \range T}\frac{1}{2}\bar{y}^{\sf T}M\bar{y} - \bar{y}^{\sf T}Nw$ subject to the constraint $\bar{y} = \tilde{y}(w) + v$. Eliminating $\bar{y}$ and writing $v = T'r$, where $T'$ is a full-column-rank matrix satisfying $\range T' = \range T$ and $r \in \real^{\bullet}$ is a new decision variable, we obtain the equivalent problem $\minimize_{r \in \real^{\bullet}} \frac{1}{2}r^{\sf T}T'^{\sf T}MT'r + r^{\sf T}T'^{\sf T}(M\tilde{y}(w)-Nw) + \tilde{y}(w)^{\sf T}\left(M\tilde{y}(w) - Nw\right)$. This unconstrained QP has a unique optimizer $r^{\star}$ if and only if $T'MT' \succ \vzeros[]$, which is equivalent to $M$ being positive definite on $\range T$.
  \end{proof} 

  \smallskip
  
\begin{lemma}[\bf Unique Dual Solution]\label{Lem:Unique-Dual}
Suppose the optimization problem \eqref{Eq:Opt-QP-OSS} has a unique primal solution $\bar{y}^{\star}$. The corresponding dual solution is unique if and only if the matrix $\left[\begin{smallmatrix}G_\perp\\H\end{smallmatrix}\right]$ is full row rank.
\end{lemma}
\begin{proof}{}
Let $\bar{y}^\star$ denote the unique primal solution of \eqref{Eq:Opt-QP-OSS}. Under our assumptions, the pair $(\lambda^\star,\mu^\star)$ is a dual solution if and only if $(\lambda^\star,\mu^\star)$ satisfies the gradient KKT condition \eqref{Eq:KKT-Gradient}, which in the present context is given by
$
\vzeros[] = M\bar{y}^\star - Nw + G_{\perp}^{\sf T}\lambda^\star + H^{\sf T}\mu^\star\,. \label{Eq:Gradient-QP}
$
The assumption of a primal solution implies that at least one dual solution $(\lambda^\star,\mu^\star)$ to the preceding exists; this solution is unique if and only if $\begin{bmatrix}G_\perp^{\sf T}&H^{\sf T}\end{bmatrix}$ is full column rank.
\end{proof}
  
We move on to the main proof; we will show that when using the FS-OM \eqref{Eq:FS-OM}, the OSS control problem is solvable if and only if the stated conditions \ref{cond:cab-stab},\ref{cond:primal-unique},\ref{cond:dual-unique} hold; a similar argument can be made for the OS-OM. A modified version of the same argument can be made for the REFS-OM when \ref{cond:cab-stab},\ref{cond:primal-unique},\ref{cond:subspaces} hold. Condition \ref{cond:cab-stab} is exactly conditions (a) and (b) of \cite[Theorem 1]{EJD-AG:75}; condition (e) of \cite[Theorem 1]{EJD-AG:75} is automatically satisfied here. We show conditions \ref{cond:primal-unique} and \ref{cond:dual-unique} are equivalent to conditions (c) and (d) of \cite[Theorem 1]{EJD-AG:75}. 
Define the matrices $\mathcal{N}$, $G$, and $G_{\perp}$ as done in Section \ref{Sec:Problem-Statement}, and without loss of generality, assume $T$ in \eqref{Eq:Def-Of-T} is selected to have full column rank.  The augmented plant using the FS-OM is
  \begin{equation*}
    \begin{aligned}
      \dot{x} &= Ax+Bu+B_ww\\
      \dot{\eta} &= \begin{bmatrix}HC\\T^{\sf T}MC\end{bmatrix}x+\begin{bmatrix}HD\\T^{\sf T}MD\end{bmatrix}u - \begin{bmatrix}L\\T^{\sf T}N\end{bmatrix}w.
    \end{aligned}
  \end{equation*}
Following \cite[Equation (13)]{EJD-AG:75}, we check whether
\begin{equation}\label{Eq:Mat-RFS}
  \mathcal{R}_{\rm FS} \define \left[\begin{array}{c|c}
	I_n & \vzeros[]\\
	\hline 
	\vzeros[] & \left[\begin{smallmatrix}
	H \\ T^{\sf T}M
	\end{smallmatrix}\right]
  \end{array}\right]\begin{bmatrix}A&B\\C&D\end{bmatrix}
\end{equation}
has full row rank. Let $\col(\alpha,\beta,\gamma) \in \nullspace \mathcal{R}_{\rm FS}^{\sf T}$, so that
\begin{equation}\label{Eq:alpha-beta-gamma}
  \begin{bmatrix}\alpha\\H^{\sf T}\beta+MT\gamma\end{bmatrix}^{\sf T}\begin{bmatrix}A&B\\C&D\end{bmatrix} = \vzeros[].
\end{equation}
Multiplying on the right by $\mathcal{N}$ and recalling that $\range \mathcal{N} = \nullspace \begin{bmatrix}A&B\end{bmatrix}$ and also that $G = \begin{bmatrix}C&D\end{bmatrix}\mathcal{N}$, we find
\begin{equation}\label{Eq:hbeta-g-1}
(H^{\sf T}\beta+MT\gamma)^{\sf T}G = \vzeros[]. 
\end{equation}
Hence, $H^{\sf T}\beta+MT\gamma \in (\range G)^\perp$. Because $(\range G)^\perp = \range G_\perp^{\sf T}$ by the definition of $G_\perp$, the above is equivalent to the existence of a vector $v$ such that
\begin{equation}\label{Eq:hbeta-gperp-1}
  H^{\sf T}\beta+MT\gamma = G_{\perp}^{\sf T}v.
\end{equation}
Recall that $\range T = (\nullspace G_{\perp}) \cap (\nullspace H)$, so $G_{\perp}T = \vzeros[]$ and $HT = \vzeros[]$. Multiplying \eqref{Eq:hbeta-gperp-1} on the left by $\gamma^{\sf T}T^{\sf T}$ we find
\begin{equation}\label{Eq:t0mt0-gamma}
  \gamma^{\sf T}T^{\sf T}MT\gamma = 0.
\end{equation}
For the sufficient direction, we show that if conditions \ref{cond:primal-unique} and \ref{cond:dual-unique} hold, then $\col(\alpha,\beta,\gamma) = \vzeros[]$. 
From condition \ref{cond:primal-unique}, it follows by Lemma \ref{Lem:Unique-Solution} that the matrix $T^{\sf T}MT$ is positive definite and hence from \eqref{Eq:t0mt0-gamma} that $\gamma = \vzeros[]$. Equation \eqref{Eq:hbeta-gperp-1} then implies that
\begin{equation}\label{Eq:Rperp-H}
  \begin{bmatrix}v&-\beta\end{bmatrix}^{\sf T}\begin{bmatrix}G_{\perp}\\H\end{bmatrix} = \vzeros[].
\end{equation}
By condition \ref{cond:dual-unique} and Lemma \ref{Lem:Unique-Dual}, the coefficient matrix \eqref{Eq:Rperp-H} in has full row rank, and hence \eqref{Eq:Rperp-H} implies that $v = \vzeros[]$ and $\beta = \vzeros[]$. Equation \eqref{Eq:alpha-beta-gamma} then implies that
$
  \alpha^{\sf T}\begin{bmatrix}A&B\end{bmatrix} = \vzeros[].
$
Since $(A,B)$ is stabilizable, the left null space of $\begin{bmatrix}A&B\end{bmatrix}$ is empty. Therefore $\alpha = \vzeros[]$, we conclude that $\mathcal{R}_{\rm FS}$ has full row rank, and thus by \cite[Theorem 1]{EJD-AG:75} the augmented plant is stabilizable/detectable; it follows by Theorem \ref{Thm:Reduction} that the OSS control problem is solvable.

For the necessary direction, we show that if either of conditions \ref{cond:primal-unique} or \ref{cond:dual-unique} fails, then we can construct $\col(\alpha,\beta,\gamma) \neq \vzeros[]$ satisfying \eqref{Eq:alpha-beta-gamma}, which in turn will violate the transmission zero condition in \cite[Theorem 1]{EJD-AG:75} and show the augmented plant is not stabilizable. Suppose \ref{cond:dual-unique} fails. Then by Lemma \ref{Lem:Unique-Dual} there exists a nonzero solution to \eqref{Eq:Rperp-H}. It cannot be the case that $\beta = \vzeros[]$, for then $v$ would be zero since $G_{\perp}$ is full row rank by construction. As a result, if we set $\gamma \define \vzeros[]$, \eqref{Eq:hbeta-g-1} implies that there exists a $\bar{\beta} \neq \vzeros[]$ such that $\bar{\beta}^{\sf T}HG = \vzeros[]$. We observe
%\begin{equation}\label{eq:abar}
$
  \bar{\beta}^{\sf T}HG = \begin{bmatrix}\bar{\beta}^{\sf T}HC&\bar{\beta}^{\sf T}HD\end{bmatrix}\mathcal{N} = \vzeros[].
  $
%\end{equation}
Since $\range \mathcal{N} = \nullspace \begin{bmatrix}A&B\end{bmatrix}$, the preceding implies that
\begin{equation}\label{eq:alphaeq}
  \begin{bmatrix}C^{\sf T}H^{\sf T}\bar{\beta}\\D^{\sf T}H^{\sf T}\bar{\beta}\end{bmatrix} \in \left(\nullspace \begin{bmatrix}A&B\end{bmatrix}\right)^\perp = \range\begin{bmatrix}A^{\sf T}\\B^{\sf T}\end{bmatrix}.
\end{equation}
As a result, a solution $\bar{\alpha}$ exists to
%\begin{equation*}
$
 \left[\begin{smallmatrix}C^{\sf T}H^{\sf T}\bar{\beta}\\D^{\sf T}H^{\sf T}\bar{\beta}\end{smallmatrix}\right] = \left[\begin{smallmatrix}A^{\sf T}\\B^{\sf T}\end{smallmatrix}\right]\bar{\alpha}.
  $
%\end{equation*}
Let $\bar{\alpha}$ satisfy the preceding. Then $\col(\alpha,\beta,\gamma) \define \col(-\bar{\alpha},\bar{\beta},\vzeros[])$ satisfies \eqref{Eq:alpha-beta-gamma}. Now suppose instead condition \ref{cond:primal-unique} fails. Then by Lemma \ref{Lem:Unique-Solution} there exists a $\bar{\gamma} \neq \vzeros[]$ such that $\bar{\gamma}^{\sf T}T^{\sf T}MT\bar{\gamma} = 0$, which by positive semidefiniteness of $T^{\sf T}MT$ implies that $T^{\sf T}MT\bar{\gamma} = \vzeros[]$, and hence that $MT\bar{\gamma} = \vzeros[]$.
%
%
%Moreover, this $\bar{\gamma}$ satisfies . To see this, note that since $M$ is positive semidefinite, $M$ has a positive semidefinite square root $P$ satisfying $M = P^{\sf T}P$ \cite{SA:15}. Hence
%\begin{equation*}
%  \begin{aligned}
%    \bar{\gamma}^{\sf T}T^{\sf T}MT\bar{\gamma} = \bar{\gamma}^{\sf T}T^{\sf T}P^{\sf T}PT\bar{\gamma} = \|PT\bar{\gamma}\|^2\,,
%  \end{aligned}
%\end{equation*}
%from which we can infer that $MT\bar{\gamma} = P^{\sf T}(PT\bar{\gamma}) = \vzeros[]$. 
It follows that the vector $\col(\alpha,\beta,\gamma) \define \col(\vzeros[],\vzeros[],\bar{\gamma})$ satisfies \eqref{Eq:alpha-beta-gamma}.

We now show that when using the OS-OM \eqref{Eq:OS-OM}, conditions \ref{cond:primal-unique} and \ref{cond:dual-unique} are equivalent to conditions (c) and (d) of \cite[Theorem 1]{EJD-AG:75}. 

The augmented plant when using the OS-OM is
  \begin{equation*}
    \begin{aligned}
      \dot{x} &= Ax+\vzeros[]\mu+Bu+B_ww\\
      \dot{\mu} &= HCx+\vzeros[]\mu+HDu-Lw\\
      \dot{\eta} &= G^{\sf T}MCx+G^{\sf T}H^{\sf T}\mu+G^{\sf T}MDu-G^{\sf T}Nw.
    \end{aligned}
  \end{equation*}
  Therefore, we examine whether the matrix
\begin{equation}\label{Eq:Mat-ROS}
\mathcal{R}_{\rm OS} \define
  \begin{bmatrix}
    A&\vzeros[]&B\\
    HC&\vzeros[]&HD\\
    G^{\sf T}MC&G^{\sf T}H^{\sf T}&G^{\sf T}MD
  \end{bmatrix}
\end{equation}
is full row rank. Let $\col(\alpha,\beta,\gamma) \in \nullspace \mathcal{R}_{\rm OS}^{\sf T}$, which is equivalent to
\begin{equation}\label{Eq:alpha-beta-gamma-ros}
  \begin{bmatrix}\alpha\\\beta\\\gamma\end{bmatrix}^{\sf T}
  \begin{bmatrix}
    A&\vzeros[]&B\\
    HC&\vzeros[]&HD\\
    G^{\sf T}MC&G^{\sf T}H^{\sf T}&G^{\sf T}MD
  \end{bmatrix} = \vzeros[].
\end{equation}
One may rewrite the above equivalently as
\begin{subequations}
  \begin{align}
    \begin{bmatrix}\alpha\\H^{\sf T}\beta+MG\gamma\end{bmatrix}^{\sf T}\begin{bmatrix}A&B\\C&D\end{bmatrix} &= \vzeros[]\label{Eq:ros-1}\\
    HG\gamma &= \vzeros[].\label{Eq:ros-2}
  \end{align}
\end{subequations}
Multiplying \eqref{Eq:ros-1} on the right by $\mathcal{N}$ and recalling that $\range \mathcal{N} = \nullspace \begin{bmatrix}A&B\end{bmatrix}$ and also that $G = \begin{bmatrix}C&D\end{bmatrix}\mathcal{N}$, we find
\begin{equation*}
  (H^{\sf T}\beta+MG\gamma)^{\sf T}G = \vzeros[].
\end{equation*}
Expanding the above yields
\begin{equation*}
\gamma^{\sf T}G^{\sf T}MG + \beta^{\sf T}HG = \vzeros[]\,.
\end{equation*}
We multiply on the right by $\gamma$ and make use of \eqref{Eq:ros-2} to find
\begin{equation}\label{eq:g0-m-g0}
  \gamma^{\sf T}G^{\sf T}MG\gamma = 0.
\end{equation}
By \eqref{Eq:ros-2} we have that $G\gamma \in \nullspace H$. By definition, $G\gamma \in \nullspace G_\perp$ also. Since $\range T = (\nullspace G_\perp) \cap (\nullspace H)$, there exists a vector $v$ such that $G\gamma = Tv$. Using \eqref{eq:g0-m-g0}, this $v$ satisfies
\begin{equation*}
  v^{\sf T}T^{\sf T}MTv = \vzeros[].
\end{equation*}
The remainder of the proof proceeds like the proof in the case of the FS-OM following equation \eqref{Eq:t0mt0-gamma}.

Finally, we show that when using the REFS-OM \eqref{Eq:REFS-OM}, condition \ref{cond:primal-unique} and the second property of condition \ref{cond:subspaces}, $(\range HG)^\perp \cap (\range T^{\sf T})^\perp = \{\vzeros[]\}$, are equivalent to conditions (c) and (d) of \cite[Theorem 1]{EJD-AG:75}.

The augmented plant using the REFS-OM is
  \begin{equation*}
    \begin{aligned}
      \dot{x} &= Ax+Bu+B_ww\\
      \dot{\eta} &= (HC+T^{\sf T}MC)x + (HD+T^{\sf T}MD)u\\
      &\quad -(Lw+T^{\sf T}N)w.
    \end{aligned}
  \end{equation*}
Therefore, we examine whether the matrix
\begin{equation}\label{Eq:Mat-RERFS}
  \mathcal{R}_{\rm RE} \define \begin{bmatrix}
  I & \vzeros[]\\
  \vzeros[] & H+T^{\sf T}M
  \end{bmatrix}
  \begin{bmatrix}
    A&B\\
    C&D
  \end{bmatrix}
\end{equation}
is full row rank. Let $\col(\alpha,\beta) \in \nullspace \mathcal{R}_{\rm RE}^{\sf T}$, which is equivalent to the equations
%\begin{equation}\label{Eq:alpha-beta-gamma-reducesrfs}
%  \begin{bmatrix}\alpha\\\beta\\\end{bmatrix}^{\sf T}
%  \begin{bmatrix}
%    A&B\\
%    HC+T^{\sf T}MC&HD+T^{\sf T}MD
%  \end{bmatrix}=\vzeros[].
%\end{equation}
%One may rewrite the above equivalently as
\begin{equation}\label{Eq:alpha-beta-gamma-reducesrfs}
    \begin{bmatrix}\alpha^{\sf T}\\\beta^{\sf T}(H+T^{\sf T}M)\end{bmatrix}\begin{bmatrix}A&B\\C&D\end{bmatrix} = \vzeros[]
\end{equation}
Multiplying on the right by $\mathcal{N}$ and recalling that $\range \mathcal{N} = \nullspace \begin{bmatrix}A&B\end{bmatrix}$ and also that $G = \begin{bmatrix}C&D\end{bmatrix}\mathcal{N}$, we find
\begin{equation}\label{eq:hgbeta}
  \beta^{\sf T}(H+T^{\sf T}M)G = \vzeros[].
\end{equation}
Hence, $H^{\sf T}\beta+MT\beta \in (\range G)^\perp$. Because $(\range G)^\perp = \range G_\perp^{\sf T}$ by the definition of $G_\perp$, \eqref{eq:hgbeta} is equivalent to the existence of a vector $v$ such that
\begin{equation}\label{eq:rerfs-v}
  H^{\sf T}\beta+MT\beta = G_\perp^{\sf T}v.
\end{equation}
Recall that $\range T = (\nullspace G_{\perp}) \cap (\nullspace H)$, so $G_{\perp}T = \vzeros[]$ and $HT = \vzeros[]$. Multiplying \eqref{eq:rerfs-v} on the left by $\beta^{\sf T}T^{\sf T}$, we find that
$\beta^{\sf T}T^{\sf T}MT\beta = 0$.

For the sufficient direction, we show that if conditions \ref{cond:primal-unique},\ref{cond:subspaces} hold then $\col(\alpha,\beta) = \vzeros[]$. By Lemma \ref{Lem:Unique-Solution}, condition \ref{cond:primal-unique} implies $M$ is positive definite on $\range T$, so it follows from the above that $T\beta = \vzeros[]$, or equivalently that $\beta \in (\range T^{\sf T})^{\perp}$. It follows then from \eqref{eq:hgbeta} that $\beta^{\sf T}HG = \vzeros[]$, implying that $\beta \in (\range HG)^{\perp}$ also. From condition \ref{cond:subspaces} we have $(\range HG)^\perp \cap (\range T^{\sf T})^\perp = \{\vzeros[]\}$, so we conclude that $\beta = \vzeros[]$. Equation \eqref{Eq:alpha-beta-gamma-reducesrfs} then reads $\alpha^{\sf T}\begin{bmatrix}A&B\end{bmatrix} = \vzeros[]$, from which we conclude $\alpha = \vzeros[]$ since $(A,B)$ is stabilizable.

For the necessary direction, we show that if any one of the conditions \ref{cond:primal-unique},\ref{cond:subspaces} fails, then we can construct a vector $\col(\alpha,\beta) \neq \vzeros[]$ satisfying \eqref{Eq:alpha-beta-gamma-reducesrfs}. Suppose \ref{cond:primal-unique} fails, so that by Lemma \ref{Lem:Unique-Solution}, there exists a $\bar{\beta} \neq \vzeros[]$ such that $\bar{\beta}^{\sf T}T^{\sf T}MT\bar{\beta} = 0$ but $T\bar{\beta} \neq \vzeros[]$. Equation \eqref{eq:hgbeta} implies that a solution $\bar{\alpha}$ exists to 
\begin{equation*}
  \begin{bmatrix}C^{\sf T}(H^{\sf T}\bar{\beta}+MT\bar{\beta})\\D^{\sf T}(H^{\sf T}\bar{\beta}+MT\bar{\beta})\end{bmatrix} = \begin{bmatrix}A^{\sf T}\\B^{\sf T}\end{bmatrix}\bar{\alpha}
\end{equation*}
using the same reasoning as in the proof in the case of the FS-OM starting at \eqref{eq:alphaeq}. With such an $\bar{\alpha}$, $\col(\alpha,\beta) \define \col(-\bar{\alpha},\bar{\beta})$ satisfies \eqref{Eq:alpha-beta-gamma-reducesrfs}.

Now suppose \ref{cond:subspaces} fails. Then there exists a $\bar{\beta} \neq \vzeros[]$ such that $T\bar{\beta} = \vzeros[]$ and $\bar{\beta}^{\sf T}HG = \vzeros[]$. The same reasoning as the proof in the case of the FS-OM starting at \eqref{eq:alphaeq} shows that a solution $\bar{\alpha}$ exists to
\begin{equation*}
  \begin{bmatrix}C^{\sf T}H^{\sf T}\bar{\beta}\\D^{\sf T}H^{\sf T}\bar{\beta}\end{bmatrix} = \begin{bmatrix}A^{\sf T}\\B^{\sf T}\end{bmatrix}\bar{\alpha}.
\end{equation*}
 With such an $\bar{\alpha}$, $\col(\alpha,\beta) \define \col(-\bar{\alpha},\bar{\beta})$ satisfies \eqref{Eq:alpha-beta-gamma-reducesrfs}.
\end{pfof}

 \begin{IEEEbiography}[{\includegraphics[width=1in,height=1.25in,clip,keepaspectratio]{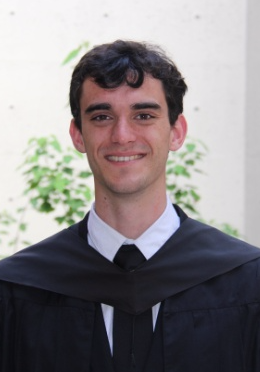}}]{Liam Lawrence} (S'17--) received the B.~Eng. degree in engineering physics from McMaster University, Hamilton, ON, Canada in 2017 and the M.A.Sc. degree in electrical engineering from the University of Waterloo, Waterloo, ON, Canada in 2019. He is currently a doctoral student in the Medical Biophysics department at the University of Toronto, Toronto, ON, Canada. He is a recipient of the NSERC Canada Graduate Scholarship at both the master's and doctoral level.
 \end{IEEEbiography}

  \begin{IEEEbiography}[{\includegraphics[width=1in,height=1.25in,clip,keepaspectratio]{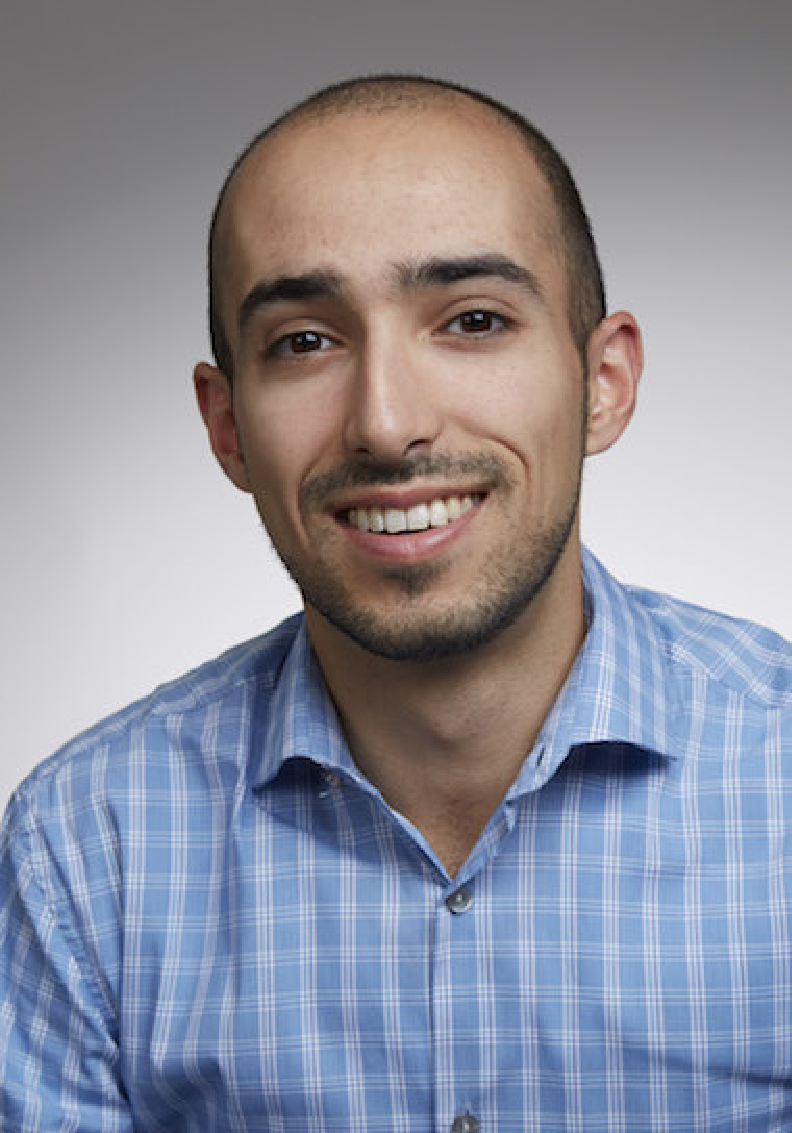}}]{John W. Simpson-Porco} (S'11--M'15--) received the B.Sc. degree in engineering physics from Queen's University, Kingston, ON, Canada in 2010, and the Ph.D. degree in mechanical engineering from the University of California at Santa Barbara, Santa Barbara, CA, USA in 2015.

He is currently an Assistant Professor of Electrical and Computer Engineering at the University of Waterloo, Waterloo, ON, Canada. He was previously a visiting scientist with the Automatic Control Laboratory at ETH Z\"{u}rich, Z\"{u}rich, Switzerland. His research focuses on feedback control theory and applications of control in modernized power grids.

Prof. Simpson-Porco is a recipient of the 2012--2014 IFAC Automatica Prize and the Center for Control, Dynamical Systems and Computation Best Thesis Award and Outstanding Scholar Fellowship. He currently serves as an Associate Editor for the IEEE Transactions on Smart Grid.
  \end{IEEEbiography}

  \begin{IEEEbiography}[{\includegraphics[width=1in,height=1.25in,clip,keepaspectratio]{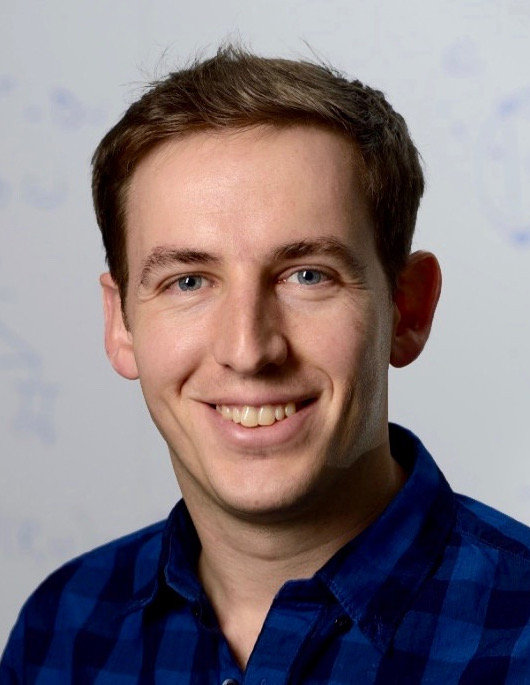}}]
 {Enrique Mallada} (S'09-M'15--) is an Assistant Professor of Electrical and Computer Engineering at Johns Hopkins University. Prior to joining Hopkins in 2016, he was a Post-Doctoral Fellow in the Center for the Mathematics of Information at Caltech from 2014 to 2016. He received his Ingeniero en Telecomunicaciones degree from Universidad ORT, Uruguay, in 2005 and his Ph.D. degree in Electrical and Computer Engineering with a minor in Applied Mathematics from Cornell University in 2014. 
 Dr. Mallada was awarded 
 the NSF CAREER award in 2018,
 the ECE Director's PhD Thesis Research Award for his dissertation in 2014, 
 the Center for the Mathematics of Information (CMI) Fellowship from Caltech in 2014,
 and the Cornell University Jacobs Fellowship in 2011. 
 His research interests lie in the areas of control, dynamical systems and optimization, with applications to engineering networks such as power systems and the Internet.
 \end{IEEEbiography}

\end{document}